\documentclass[a4paper,12pt,reqno]{amsart}
\usepackage{latexsym}
\usepackage{amsmath,amsthm,amssymb,amscd}
\usepackage{fullpage}
\usepackage{xcolor}
\usepackage{parskip}
\usepackage{mathtools}
\usepackage{thmtools}
\mathtoolsset{showonlyrefs}
\usepackage{csquotes}
\usepackage{comment}
\usepackage[activate={true,nocompatibility},final,tracking=true,kerning=true,spacing=true]{microtype}
\usepackage[pdfencoding=unicode,pdftex,bookmarks=true,bookmarksnumbered]{hyperref}
\definecolor{citecolour}{rgb}{0.0, 0.0, 0.8}
\colorlet{linkcolour}{green!50!black}
\hypersetup{colorlinks,breaklinks,
		linkcolor=linkcolour,citecolor=citecolour,
		filecolor=linkcolour, urlcolor=linkcolour}

\usepackage{txfonts,pxfonts,tikz} 
\usepackage[T1]{fontenc}

\newtheorem{prevtheorem}{Theorem}

\newtheorem{theorem}{Theorem}

\newtheorem{proposition}{Proposition}

\newtheorem{remark}{Remark}
\newtheorem{lemma}[theorem]{Lemma}
\newtheorem{corollary}{Corollary}

\newenvironment{proofofCorA}{{\bf {Proof of Corollary \ref{Cor:dl}.} }}{\hfill $\blacksquare$ \\}

\def\Z{\mathbf{Z}}
\def\C{\mathbf{C}}
\def\N{\mathbf{N}}

\def\D{\mathcal{D}}
\def\dl{\mathrm{dl}}
\def\PSL{\mathrm{PSL}}
\def\SL{\mathrm{SL}}
\def\Aut{\mathrm{Aut}}
\def\Syl{\mathrm{Syl}}
\usepackage{graphicx} 

\begin{document}

\title{Groups with 5 nontrivial conjugacy classes of non self-normalizing subgroups are solvable }
\author{Maria Loukaki }
\address{Department of Mathematics \& Applied Mathematics, University of Crete, Greece}
\email{mloukaki@uoc.gr}
\date{February 2025}
\keywords{non self-normalizing, solvable groups}
\subjclass[2010]{20E07, 20E34, 20D10, 20D15}

\begin{abstract}
For any  $n$  nonnegative integer,  a family of groups, denoted by $ \mathcal{D}_n $, was introduce by Bianchi et al.,  as the collection of all  finite groups  with exactly $n$ conjugacy classes of nontrivial, non self-normalizing subgroups. It was conjectured that $\mathcal{D}_5$  consists of solvable groups with derived length at most $3$. In this note we verify their conjecture.

\end{abstract}

\maketitle

\section{Introduction}
For any $n$ nonnegative integer,  let $ \D_n $  denote the family of all finite groups  with exactly $n$ conjugacy classes of nontrivial, non self-normalizing subgroups. These families  were  introduced by Bianchi et al. in \cite{Bianchi}, and their properties regarding solvability and nilpotency have been examined. Several interesting  results were established, including a complete classification of the groups lying in $\D_i$ for $i=0,1, 2, 3$. Additionaly, it was shown that  the families $\D_i$  contain only solvable groups for all $i=0,1,2,3$ and $4$ with one exception: $A_5 \in \D_4$. Moreover, the derived length $\dl(G)$ of any solvable $G \in D_i$,  for $i=0, 1, \cdots, 4$, is  at most  $2$ unless $G \cong \SL_2(3)$. 

In the same paper, the authors conjectured that any \( G \in \D_5 \) is solvable with \( \dl(G) \leq 3 \) which is confirmed by our main result:
\begin{prevtheorem}\label{Thm:A}
    If  a finite group $G$ lies in $\D_5$ then $G$ is solvable with derived length   at most $3$.  
\end{prevtheorem}

As an immediate consequence we get the following.
\begin{corollary}\label{Cor:dl}
 Assume  $n \geq 5$ and  $G \in \D_n $ is a  solvable finite group. Then $\dl(G) \leq n-2$.
\end{corollary}

We follow the notation used in \cite{Bianchi}. In particular, we denote by $\D(G)$ the number of conjugacy classes of nontrivial subgroups of $G$ that are not self-normalizing. Moreover,  if $N \unlhd G$ we write $\D_G(N)$  for the number of $G$-conjugacy classes of nontrivial subgroups of $G$ properly  contained in $N$ that are not self-normalizing.  We also write $\varpi(G)$ for  the set of primes that divide the order of $G$.

\section{Proofs}
Before giving the proof of our theorem  we recall some useful  results from \cite{Bianchi}. The following is Lemma 3.2 there.
\begin{lemma}\label{Lem:D(G/N)}
    If $1 \neq N \lhd G$ then $\D(G/N) \leq \D(G)- \D_G(N) -1 \leq \D(G) - 1$. 
\end{lemma}

Several of the results proven in \cite{Bianchi} for $\D_i$ with $i \leq 4$,  are collected in the next proposition that we will use often. 
\begin{proposition}\label{Prop:D_{0-4}}
Assume $G$ is a finite group. 
    \begin{itemize}
        \item If $G \in \D_0$ if and only if  $G$ is cyclic of prime order. (Proposition 3.1 in  \cite{Bianchi}).
        \item If $G \in \D_n$ with $n \leq 3$ then $G$ is solvable with  derived length at most $2$. (Proposition 3.3 in  \cite{Bianchi}). 
        \item The only non-solvable group in $\D_4$ is the alternating group $A_5$.(Theorem 8.2  in \cite{Bianchi}). 
        \item If $G $ is a solvable group in $\D_4$ then either $\dl(G) \leq 2$ or $G \cong \SL_2(3)$. (Theorem 8.3 in \cite{Bianchi}).
    \end{itemize}
\end{proposition}

We will now prove our theorem, beginning with a straightforward lemma.

\begin{lemma}\label{Lem:D(G/N)=D(G)-1}
If $N \unlhd G$ with $\D(G/N) =\D(G)-1$ then $\D_G(N)= \D(N)=0$ and $N$ is a cyclic group of prime order.
\end{lemma}

\begin{proof}
    By \autoref{Lem:D(G/N)} we have $\D(G)-1=\D(G/N) \leq\D(G)- \D_G(N)-1$. Hence $\D_G(N)=0$ and so every nontrivial proper subgroup  of $N$ is selfnormalizing. Thus $\D(N)=0$ and $N$ is cyclic of prime order by \autoref{Prop:D_{0-4}}. 
\end{proof}

The next observation is useful. 
\begin{remark}\label{Rem:NnormalG}
Let $G \in \D_5$ and $N$ be a proper normal subgroup of $G$. For any Sylow subgroup $P$ of $N$, Frattini's Lemma implies that $\N_G(P) \nsubseteq N$. Consequently, for every $p \in \varpi(N)$, the corresponding Sylow $p$-subgroup $P$ of $N$  is not self-normalizing in $G$. Thus, $\left\lvert\varpi(N)\right\rvert \leq 4$ (since  $N$ itself is also not self-normalizing). If $\left\lvert P\right\rvert=p^a$, then in addition to $P$, $N$ includes non-selfnormalizing, non-$G$-conjugate subgroups of orders $p$, $p^2$, ..., $p^{a-1}$. As a result, $N$ falls into one of the following cases:
\begin{itemize}
    \item $N$ is a $Z$-group, meaning every Sylow subgroup of $N$ is cyclic. In this scenario, both $N/N'$ and $N'$ are cyclic groups (see Theorem 5.16 in \cite{Isaacs}), implying that $N$ is solvable with derived length $\dl(N) \leq 2$.
    \item $\left\lvert\varpi(N)\right\rvert=1$ and thus $N$ is  nilpotent of order $\left\lvert N\right\rvert=p^a$ for some prime $p$ and $a \leq 5$.
    
\item  $\left\lvert\varpi(N)\right\rvert=2$  and $|N|$ equals one of  $p^2q^2,\,  p^3q$ or $p^2q$ where $p$ and $q$ are distinct primes. In this case $N$ is solvable.

\item $\left\lvert\varpi(N)\right\rvert=3$ and $\left\lvert N\right\rvert=p^2qr$, where $p, q, r$ are distinct primes.
\end{itemize}
\end{remark}

Any group \( G \in \D_5 \) has no non-solvable proper normal subgroups, as demonstrated by the following proposition.

\begin{lemma}\label{Lem:N.nonsolv}
If $G$ is a finite group and $\D(G)=5$ then any proper normal subgroup of $G$ is solvable. 
\end{lemma}

\begin{proof}
Assume the opposite and let $N$ be a proper normal non-solvable subgroup of $G$. 
In view of \autoref{Rem:NnormalG}, $ \left\lvert N\right\rvert= p^2qr$ with $p, q, r$ being distinct primes and one of them is $2$ (or else $N$ is solvable).  Let $P$, $Q$, and $R$ be the Sylow $p$, $q$, and $r$-subgroups of $N$, respectively. The subgroup $P_1$ of order $p$ within $P$ is not self-normalizing in $N$ (and  in $G$). Therefore, the groups $P_1$, $P$, $Q$, $R$, and $N$ represent the five distinct $G$-conjugacy classes of nontrivial, non self-normalizing subgroups of $G$.
Thus, $N$ does not contain any other nontrivial  non self-normalizing subgroup of $G$ besides $P_1$, $P$, $Q$, and $R$
and $\D_G(N) = 4$.

Assume first that  $q=2$ and so  $\left\lvert N\right\rvert=2p^2r$. The Sylow $2$-subgroup of $N$ is cyclic, and by Corollary 5.14 in \cite{Isaacs}, $N$ has a normal $2$-complement, denoted as $S$. Then $|S|= p^2r$ and $S$   is not $G$-conjugate to any of $P_1, P, Q, R$ and $N$, contradicting our assumption.

Assume now that $p=2$ and so $ \left\lvert N\right\rvert= 4qr$. Then every Sylow subgroup of $N$ of odd order is cyclic while the Sylow $2$-subgroup of $N$ is of order $4$.  According to Theorem A in \cite{Suzuki} we get that $N \cong \PSL(2,s) \times K$, for some $Z$-group  $K$, and $s \geq 5$ because  $N$ is non-solvable. 
Since $4qr = \left\lvert \PSL(2,s) \times K \right\rvert$, we conclude that $s$ is an odd prime, and without loss of generality  $s=q$, and  $K$ is  trivial. Hence
 \[
 4qr= |N|=\left\lvert \PSL(2,q)  \right\rvert = \frac{q(q^2-1)}{2}.
 \] 
Thus, $2r = \frac{q-1}{2} \cdot \frac{q+1}{2}$ leads to $q=5$ and $r=3$, which gives $N \cong \PSL(2,5) = A_5$. Since $\D_G(N) = 4$, we have $\D(G/N) \leq \D(G) - \D_G(N) - 1 = 0$. Consequently, $G/N$ is a cyclic group of prime order.  If $C_G(N) \neq 1 $ and $  x \in C_G(N)$, then $ \langle x \rangle$  would serve as a normal complement of $ N $ in $ G$, contradicting the fact that \( \D(G) = 5 \). The $N/C$ theorem implies that $G$ is isomorphic to a subgroup  of $\Aut(N)  \cong S_5$ and so  $G \cong S_5$. But $\D(S_5)= 13$ and this final contradiction completes the proof of  the lemma. 
\end{proof}

We  next prove  that no simple group is contained in $\D_5$, while noting that, in contrast, $\PSL(2,8)) \in \D_6$.

\begin{proposition}\label{Lem:Simple}
    Assume $G$ is a non-abelian simple group $G \ncong A_5$. Then $\D(G) \geq 6$.
\end{proposition}
\begin{proof}
    In view of the work done in \cite{Bianchi} (see \autoref{Prop:D_{0-4}}) 
    it suffices to show that there is no simple group in $\D_5$. Assume the opposite for contradiction, and let $G \in \D_5$ be simple and non-abelian. 
So $\left\lvert\varpi(G)\right\rvert \geq 3$ and $2 \in \varpi(G)$.

Let $p \in \varpi(G)$ and  $P$ be any Sylow $p$-subgroup of $G$. If  $P$ is abelian then $\N_G(P) >P$, or else $P \leq \mathbf{Z}(\N_G(P))$ and thus $G$ has a normal $p$-complement, by Theorem 5.13 in \cite{Isaacs}. Observe now that whenever $|P| \geq p^2$ the prime $p$ adds at least $+2$ to $\D(G)$ due to the presence of subgroups of order $p$ and $p^2$  (even if $|P|= p^2$ since  $P$ itself in not self-normalizing in this case). If $|P|=p$  then the prime  $p$  contributes $+1$ to $\D(G)$.
Corollary 5.14 in \cite{Isaacs} ensures  that the Sylow $2$-subgroups of $G$ are non-cyclic, contributing at least $+2$ to $\D(G)$.

We claim that $\left\lvert\varpi(G)\right\rvert < 5$. Indeed, if $p_1, \ldots, p_4$ are four distinct odd primes in $\varpi(G)$, each contributes at least $+1$ to $\D(G)$, while the prime $2$ contributes at least $+2$. This results in $\D(G) \geq 6$, contradicting our assumption.
Hence $3 \leq \left\lvert\varpi(G)\right\rvert \leq 4$. 

Assume first that $\left\lvert\varpi(G)\right\rvert = 3$. Then $G$ is a simple group whose order is divisible by three primes only and thus is isomorphic to one of the following
\[
A_5,\,  A_6,\,  \PSL(2,7), \, \PSL(2,8),\,  \PSL(2,17), \, \PSL(3,3), \, \mathrm{U}_3(3), \, \mathrm{U}_4(2), 
\]
according to \cite{Herzog}. It can easily be checked by GAP that
$
\D(A_6)= 11, \, \D(\PSL(2,7))=8, \,  \, \D(\PSL(2,8))=6 , \, \, \D( \PSL(2,17))= 14, \, \, \D(\PSL(3, 3))= 38,  \, \, \D(\mathrm{U}_3(3)) = 35 $ and  $\D(\mathrm{U}_4(2)) = 96$.    As none of them lies in $\D_5$ we get that 
$\left\lvert\varpi(G)\right\rvert = 4$. As the  prime $2$ contributes  at least $+2$ in $\D(G)$, each one of the remaining three odd primes must contribute exactly $+1$ in $\D(G)$.  Hence 
\[
\left\lvert G \right\rvert = 2^apqr, 
\]
with $2\leq a \leq 3$. In particular $G$ has cyclic Sylow subgroups for each odd prime. 

{\bf Case 1.} $a=2$ and let $r < q < p$. 
From Theorem on pg.671 in \cite{Suzuki} we have   $G \cong \PSL(2,p)$ and so  $ 4pqr= \left\lvert G \right\rvert= \frac{p(p^2-1)}{2}$. 
As $3 \big|  (p-1)(p+1)$ for every prime $p\neq 3$  we get  $r=3$ and 
\[
12pq= \left\lvert G \right\rvert= \frac{p(p^2-1)}{2}. 
\]
Thus, we have $24q = p^2 - 1$ with $q < p \geq 7$. Checking case by case the divisors of $24q$ reveals two possibilities: either $p - 1 = 2q$ and $p + 1 = 12$, or vice versa, $p - 1 = 12$ and $p + 1 = 2q$. This leads to the two groups $\PSL(2,11)$ and $\PSL(2,13)$. Using GAP, we find $\D(\PSL(2,11)) = 8$ and $\D(\PSL(2,13)) = 10$. Therefore, Case 1 cannot occur. 

{\bf Case 2.} $a=3$ and let $A$ be a Sylow $2$-subgroup of $G$ with $|A|=8$. Then $\N_G(A) = A$  or else $\N_G(A) >A$ and so $G$ would have non self-normalizing subgroups of order $2, 4$ and $8$, leaving no room for the non self-normalizing subgroups provided by the other three odd primes.
Also $A$ is non-abelian or else $A \leq \Z(\N_G(A))$ forcing $G$ to have a normal $2$-complement.

Hence $A$ is either isomorphic to the quaternion  or the dihedral.  According to the  Brauer-Suzuki Theorem (see \cite{Bra-Suz}) there  is no simple group having the quaternion group as a Sylow 2-subgroup. So  $A\cong D_4$ and Theorem B in \cite{Suzuki} applies. As $\D(PSL(2,8)= 6$ we get that  $G \cong \PSL(2,s)$ for some odd $s$ and without loss may assume $s=p$ with $p \geq 5$. Hence 
\[
8pqr= \left\lvert G \right\rvert = \frac {p(p^2-1)}{2}.
\]
Since \( 3 \mid (p^2-1) \), we can set \( r=3 \), leading to the equation 
\[
48q= (p-1)(p+1).
\]
As earlier we check the divisors of $48q$ and we end up only with one possible pair $(p-1, p+1)=(2q, 24)$. Therefore  $G \cong \PSL(2,23)$ with   $\D(\PSL(2,23))=15$. Thus Case 2 is also excluded confirming the lemma.
\end{proof}

We are ready now to prove that every group in $\D_5$ is solvable.
\begin{theorem}\label{Thm:D_5-is-solvable}
    Assume $G$ is a finite group with $\D(G)=5$. Then $G$ is solvable. 
\end{theorem}

\begin{proof}
Let $G$ be a counterexample, that is,  $G \in \D_5$ is non-solvable.  Then 
    $G$ is not abelian and is not  simple according to  \autoref{Lem:Simple}, so let $N$ be a proper normal subgroup of $G$. 
     By \autoref{Lem:N.nonsolv},  $N$ is solvable, forcing $G/N$ to be non-solvable. We have  $\D(G/N) \leq \D(G) -1 = 4$ and the only non-solvable group in $\D_i$ for $i \leq 4$,  is $A_5$. Consequently,
 $G/N \cong A_5$ and $\D(G/N)= 4$ and so (see \autoref{Lem:D(G/N)=D(G)-1})   $N\cong C_s$ is a cyclic group of prime order $s$. Consider the centralizer $\C_G(N)$ of $N$ in $G$. Since $G/N$ is simple and $\C_G(N)$ is normal in $G$, either $\C_G(N) = N$ or $\C_G(N) = G$. The first option is impossible due to the $N/C$ theorem (otherwise $A_5 \cong G/N$ would be isomorphic to a subgroup of $\Aut(C_s) \cong C_{s-1}$). Therefore, $N$ must be central in $G$, in fact, $N = \Z(G)$, as the previous argument applies to $\Z(G)$ as well.

The prime $s$ is different from at least one of $3, 5$ and without loss we may assume $s \neq 3$. 
Write  $\bar{A}, \bar{B}, \bar{C}, \bar{D} $ for  the representatives of the four nontrivial non self-normalizing subgroups of $G/N$ with  orders 
$3, 5, 4$ and $2$ respectively. If $A, B, C, D$ are their preimages in $G$ then these four along with $N$  form a set of representatives of the five classes of non self-normalizing subgroups of $G$ with orders $3s, 5s, 4s, 2s$ and $s$ respectively. If $X \in \Syl_3(A)$ (and thus $X \in \Syl_3(G)$) then we clearly have $A = N \times X$ and $N_G(X) \geq A > X$. Hence $X$ is not self-normalizing and is not $G$-conjugate  to any of the other five subgroups of $G$, implying  $\D(G) \geq 6$. 

This final contradiction proves that no such $G$ exists and the theorem follows.
\end{proof}

The final component of our main theorem regarding  the derived length of groups in $\D_5$ is addressed in the next theorem.
\begin{theorem}\label{Th:dl}
    If $G \in \D_5$ then its  derived length $\dl(G) $ satisfies $\dl(G) \leq 3$.
\end{theorem}

\begin{proof}
Let $G \in \D_5$ and assume $G$  is non-abelian.  As we have seen $G$ is solvable, so  $1< G' < G$.
We apply \autoref{Rem:NnormalG} for $G'$ in the place of $N$, so $\left\lvert\varpi(G')\right\rvert \leq 4$. 
We distinguish four cases. 

{\bf Case 1:}   $G'$ is a $Z$-group and thus $\dl(G') \leq 2$ as wanted.

{\bf Case 2:}  $\left\lvert\varpi(G')\right\rvert=1$.   Then $\left\lvert G'\right\rvert=p^a$ for some prime $p$ and $a \leq 5$. If the derived length of $G'$ is  $k + 1$, it was shown by
Hall (see III.7.11 in \cite{Huppert}) that $a \geq 2^k + k$. Hence $k \leq 1$ and $\dl(G') \leq 2$ as desired.

{\bf Case 3:} $\left\lvert\varpi(G')\right\rvert=2$ and  $\left\lvert G'\right\rvert$ is one of  $p^2q^2$, $p^3q$, or $p^2q$, where $p$ and $q$ are distinct primes. If $\left\lvert\varpi(G')\right\rvert= p^2q$ then $G'$ has  a normal Sylow subgroup (see Theorem 1.31 in \cite{Isaacs}), so  $G^{(2)}$ is either trivial or abelian  and thus $\dl(G')\leq 2$. 

If $\left\lvert\varpi(G')\right\rvert= p^2q^2$, let $P$ and $Q$ be Sylow $p$ and $q$ subgroups of $G'$, with $P_1 < P$ and $Q_1 < Q$ subgroups of orders $p$ and $q$, respectively. Then $P, P_1, Q, Q_1$ along with $G'$ represent the five classes of non self-normalizing subgroups of $G$. If $G^{(2)} \neq 1$, it must be one of $P$, $P_1$, $Q$, or $Q_1$, since it is normal in $G$. In any case, $G^{(2)}$ is abelian, leading to $G^{(3)}=1$ as required.

If $\left\lvert\varpi(G')\right\rvert= p^3q$  let $P$, $Q$ be  Sylow $p$ and $q$-subgroups of $G'$ and   $P_1 < P_2  < P$ subgroups of order $p$ and $p^2$, respectively,  within $P$. As before, the groups $P_1, P_2, P, Q$ and $G'$ represent the five classes of non self-normalizing subgroups of $G$. Therefore, $G^{(2)}$ must either be one of them  or be trivial. All these groups are abelian, with the potential exception of $P$ that could be an extra special group of order $p^3$. In this scenario, the center $\Z(P)$ is a normal subgroup of $G$ of order $p$, and any non-central element of $P$ with order $p$ generates another non self-normalizing subgroup of $G$ that is not conjugate to $\Z(P)$. This leads to $\D(G) \geq 6$, resulting in a contradiction. Thus, $G^{(2)}$ must  be abelian.

{\bf Case 4:} $\left\lvert\varpi(G')\right\rvert=3$ and 
 and $\left\lvert G'\right\rvert=p^2qr$, where $p, q, r$ are distinct primes.
As earlier, if  $P, Q, R$ are   Sylow $p, q$ and $r$-subgroups of $G'$ and   $P_1 < P $ a subgroup of order $p$ then $P_1, P, Q, R $ and $G'$ represent the five classes of non self-normalizing subgroups of $G$. Hence $G^{(2)}$ must be one of them and hence it is abelian. 

This final case completes the proof of the theorem.    
\end{proof}

 \autoref{Cor:dl} was proposed in \cite{Bianchi} and improves their earlier bound of $\log_2(n/2)+1$ from Theorem 3 in \cite{Bianchi}. Its proof follows easily  from \autoref{Thm:A}.
 
\begin{proofofCorA}
  Induct on $n$, with the base case $n=5 $,   being trivially true by \autoref{Th:dl}. For $n \geq 6$, let  $N$ be  a minimal normal subgroup of $G$. Then $N$ is abelian while $\D(G/N) \leq n-1$, and so, $\dl(G/N) \leq n-3$ and $\dl(G) \leq n-2$.      
\end{proofofCorA}

\bibliographystyle{amsalpha}
\bibliography{Bibliography}
\end{document}